\documentclass[a4paper,10pt]{amsart}

\usepackage{amsfonts, latexsym, amsmath, amssymb, amsthm, amscd}
\usepackage[all]{xy}
\usepackage[english]{babel}
\usepackage[utf8]{inputenc}
\usepackage{graphicx}
\usepackage{booktabs}
\usepackage{array, tabularx}
\usepackage{textcomp}
\usepackage{tikz}
\usepackage{multirow}
\usepackage{paralist}
\usepackage{comment}
\usepackage{url}
\usepackage{tensor}
\usepackage{MnSymbol}
\usepackage[hypertexnames=false,
backref=page,
    pdftex,
    pdfpagemode=UseNone,
    breaklinks=true,
    extension=pdf,
    colorlinks=true,
    linkcolor=blue,
    citecolor=blue,
    urlcolor=blue,
]{hyperref}

\usepackage[top=1.5in, bottom=.9in, left=0.9in, right=0.9in]{geometry}

\newcommand{\referenza}{}

\newtheorem{thm}{Theorem}[section]
\newtheorem*{thm*}{Theorem \referenza}
\newtheorem{cor}[thm]{Corollary}
\newtheorem*{cor*}{Corollary \referenza}
\newtheorem{lem}[thm]{Lemma}
\newtheorem*{lem*}{Lemma \referenza}

\newtheorem*{prop*}{Proposition \referenza}

\newtheorem*{conj*}{Conjecture \referenza}
\theoremstyle{definition}
\newtheorem{rmk}[thm]{Remark}

\theoremstyle{definition}

\newtheorem{defi}[thm]{Definition}

\DeclareFontFamily{U}{MnSymbolC}{}
\DeclareSymbolFont{MnSyC}{U}{MnSymbolC}{m}{n}
\DeclareFontShape{U}{MnSymbolC}{m}{n}{
    <-6>  MnSymbolC5
   <6-7>  MnSymbolC6
   <7-8>  MnSymbolC7
   <8-9>  MnSymbolC8
   <9-10> MnSymbolC9
  <10-12> MnSymbolC10
  <12->   MnSymbolC12}{}
\DeclareMathSymbol{\intprod}{\mathbin}{MnSyC}{'270}

\DeclareMathOperator{\Ker}{Ker}

\DeclareMathOperator{\im}{Im}
\DeclareMathSymbol{\Finv} {\mathord}{AMSb}{"60}

\newcommand{\Q}{\mathbb{Q}}

\newcommand{\R}{\mathbb{R}}
\newcommand{\C}{\mathbb{C}}

\newcommand{\del}{\partial}
\newcommand{\delbar}{\overline{\partial}}
\newcommand{\N}{\mathbb{N}}
\numberwithin{equation}{section}

\newcommand{\End}{End}
\newcommand{\A}{\mathcal{A}}
\let\c\overline
\let\phi\varphi
\newcommand{\de}[2]{\frac{\partial #1}{\partial #2}}

\allowdisplaybreaks

\title[Deformations of Strong K\"ahler with torsion metrics]{Deformations of Strong K\"ahler with torsion metrics}

\author{Riccardo Piovani}
\address[Riccardo Piovani]{
Dipartimento di Matematica\\
Universit\`a di Pisa}
\email{riccardo.piovani@phd.unipi.it}

\author{Tommaso Sferruzza}
\address[Tommaso Sferruzza]{
Dipartimento di Scienze Matematiche, Fisiche e Informatiche\\
Università di Parma}
\email{tommaso.sferruzza@unipr.it}

\keywords{SKT metrics, deformations of complex structures}
\thanks{The first author is partially supported by GNSAGA of INdAM}
\subjclass[2020]{32G05, 53B35, 53C55}%{53C55, 53C44, 32J15, 57T10} %16E45

\date{\today}

\begin{document}

\begin{abstract}
Existence of strong K\"ahler with torsion metrics, shortly SKT metrics, on complex manifolds has been shown to be unstable under small deformations. 
We find necessary conditions under which the property of being SKT is stable for a smooth curve of Hermitian metrics $\{\omega_t\}_t$ which equals a fixed SKT metric $\omega$ for $t=0$, along a differentiable family of complex manifolds $\{M_t\}_t$.
\end{abstract}

\maketitle

\section{Introduction}
Let $(M,J,g,\omega)$ be an Hermitian manifold. If the fundamental form $\omega$ of $g$ is closed,  i.e., $d\omega=0$, where $\omega(\cdot,\cdot)=g(J\cdot,\cdot)$, the metric $g$ is said to be K\"ahler. 
By the celebrated theorem of Kodaira and Spencer, see \cite{KS60}, we know that on a compact complex manifold the K\"ahler condition,  i.e., the property of admitting a Hermitian metric with closed fundamental form, is stable under small deformations of the complex structure. Therefore, it is straightforward to consider notions that generalize the K\"ahler condition which naturally arise in the Hermitian setting and study their stability under deformations.

When the fundamental form $\omega$ (or its powers) belong to kernel of certain differential operators deriving from the complex structure, special Hermitian structures arise, e.g., SKT and balanced metrics. More precisely, denoting $(M,J)$ a complex manifold of complex dimension $n$, a Hermitian metric $g$ on $(M,J)$ with fundamental associated form $\omega$ is said to be \emph{strong K\"ahler with torsion}, shortly SKT, or \emph{pluriclosed} if $\del\delbar\omega=0$. Note that if $g$ is K\"ahler, then it is also trivially SKT.
Another notion which generalizes K\"ahlerness is the balanced condition,  i.e., $d\omega^{n-1}=0$.
%It is in fact a particular case of a \emph{p-K\"ahler} metric, i.e., $d\omega^p=0$, for $2\le p \le n-1$.
Also in this case, if $g$ is K\"ahler, then it is balanced. In respectively \cite{FT09} and \cite{AB90}, it is proved that the existence of SKT and balanced metrics is not stable, once the base compact complex manifold is deformed via a smooth family of complex structures. In both works the authors construct explicit examples of differentiable families of complex manifolds which do not admit respectively SKT and balanced metrics.

Since the existence of SKT metrics on complex manifolds is not stable under deformations, it is worth investigating under which assumptions a SKT metric exists on a deformed complex manifold. 
More in detail, we will be interested in studying SKT metrics which are not K\"ahler, taking into account the stability result of the K\"ahler condition by Kodaira and Spencer. Analogously to the K\"ahler setting as studied in \cite{HL}, the existence of SKT metrics on compact complex manifolds can be intrinsically characterized in terms of currents, see \cite{E}.
SKT metrics play a relevant role in the following setting. Let $(M,J,g,\omega)$ be a Hermitian manifold of complex dimension $n$. It is known that there exists a unique connection $\nabla^B$, called \emph{Bismut connection}, satisfying $\nabla^Bg=0$, $\nabla^BJ=0$ for which $g(X,T(Y,Z))$ is totally skew-symmetric, where $T$ denotes the torsion of $\nabla^B$. The resulting $3$-form turns out to be equal to $Jd\omega$. The properties of such connection are related to what is called \emph{K\"ahler with torsion geometry} (we refer to \cite{FG}, \cite{Ga97}, \cite{Stro} for further details), and if $Jd\omega$ is closed, or equivalently if $\del\delbar\omega=0$, then the Hermitian structure is strong K\"ahler with torsion and $g$ is indeed called SKT. We point out that compact complex manifolds admitting SKT structures have been proven to be valid candidates for the study of generalizations of the K\"ahler-Ricci flow, see for example \cite{ST09}. See also \cite{Cav} for a development of Hodge theory on SKT manifolds by tools from generalized complex geometry.

The theory regarding compact complex manifolds admitting SKT metrics in complex dimension $n$ at least three is completely different from the one on compact complex surfaces. Indeed, on a compact complex surface a Hermitian metric is SKT if and only if it is \emph{Gauduchon}, i.e., $\del\delbar\omega^{n-1}=0$, and it is well known, by a remarkable result of Gauduchon in \cite{Ga}, that there exists a Gauduchon metric in the conformal class of any given Hermitian metric on a compact complex manifold.
Therefore on a compact complex surface the SKT condition is stable under small deformations
of the complex structure.

Examples of compact complex manifolds admitting SKT metrics of complex dimension at least three are given by nilmanifolds, i.e., compact quotients of connected simply-connected nilpotent Lie groups by uniform discrete subgroups. In particular, for complex dimension three, or real dimension six, nilmanifolds with SKT metrics have been characterized. In \cite{FPS04}, Fino, Parton and Salamon prove that if $M$ is a $6$-dimensional nilmanifold with an invariant complex structure $J$, then the SKT condition is satisfied by either all invariant Hermitian metrics or by none; moreover, it is satisfied if and only if the complex structure $J$ fulfils a suitable property.
Therefore, it is worth studying what happens in higher dimensions. 

In \cite{RT12}, Rossi and Tomassini prove that if $M$ is a $8$-dimensional nilmanifold with an invariant complex structure $J$, then the SKT condition is satisfied by all invariant Hermitian metrics if and only if the complex structure $J$ fulfils a suitable property. Therefore, it can happen that a $8$-dimensional nilmanifold with an invariant complex structure $J$ admits both invariant SKT metrics and invariant non-SKT metrics.

In this paper, we prove the following necessary condition to the existence of a smooth family of SKT metrics on a differentiable family of complex manifolds.
\begin{thm}\label{thm:main}
Let $(M,J,g,\omega)$ be a compact Hermitian manifold with $g$ a SKT metric. Let $\{M_t\}_{t\in I}$ be a differentiable family of compact complex manifolds parametrized by the $(0,1)$-vector form $\phi(t)$, for $t\in I=(-\epsilon,\epsilon)$, $\epsilon>0$. Let $\{\omega_t\}_{t\in I}$ be a smooth family of Hermitian metrics on each $M_t$ written as
\begin{equation*}
\omega_t=e^{i_{\phi(t)}|i_{\overline{\phi(t)}}}\,\,(\omega(t)),
\end{equation*}
where $\omega(t)$ has local expression $\omega_{ij}(t)\, dz^i\wedge d\overline{z}^j\in\A^{1,1}(M)$. Denote by $\omega^\prime(t):=\de{}{t}\omega_{ij}(t)\, dz^i\wedge d\overline{z}^j\in\A^{1,1}(M)$. Then, if the metrics $\omega_t$ are SKT for every $t\in I$, the following condition must hold
\begin{equation}\label{eq:main}
2i\mathfrak{Im}(\del\circ i_{\phi'(0)}\circ \del)(\omega)=\del\delbar\omega'(0).
\end{equation}
\end{thm}
Here, for any $p,q$ and for $t\in(-\epsilon,\epsilon)$, the map $e^{i_{\phi(t)}|i_{\overline{\phi(t)}}}\colon\A^{p,q}(M)\rightarrow\A^{p,q}(M_t)$ is a real linear isomorphism between the space of $(p,q)$-forms on $M$ and the space of $(p,q)$-forms on $M_t$, called \emph{extension map}; see equation (\ref{def-exp}) for its definition. By $i_\psi$ we denote the contraction operator on $(p,q)$-forms by the $(0,1)$-vector form $\psi$; see section \ref{preliminaries} for its definition.
As a consequence, we have the following cohomological obstruction.
\begin{cor}\label{cor:main}
Let $(M,J,g,\omega)$ be a compact Hermitian manifold. A necessary condition for the existence of a smooth family of SKT metrics which equals $\omega$ in $t=0$ along the family of deformations $t\mapsto\phi(t)$ is that the following equation must hold
\[
\left[\mathfrak{Im}(\del\circ i_{\phi'(0)}\circ \del)(\omega)\right]_{H_{BC}^{2,2}(M)}=0.
\]
\end{cor}
Here, $H_{BC}^{p,q}(M)$ denotes the Bott-Chern cohomology group of bi-degree $(p,q)$ defined on the complex manifold $M$.

We remark that our results involves a slightly different notion of stability of SKT metrics from the usual one. Our results concern the existence of smooth families of SKT metrics $\{\omega_t\}_t$ on the differentiable family of complex manifold $\{M_t\}_t$,
%, such that $\omega_0$ is a fixed SKT metric on the base complex manifold $M_0$
and do not concern the existence of SKT metrics on $\{M_t\}_t$ in full generality.

To prove our result, following Rao and Zhao in \cite{RZ}, we develop a method to compute the complex differentials $\del_t$ and $\delbar_t$ acting on $(p,q)$-forms on a differentiable family of complex manifolds $\{M_t\}_t$, which depends on the complex differentials $\del_0=\del$ and $\delbar_0=\delbar$ on the base complex manifold $M_0=M$, and on the $(0,1)$-differential form with values in the holomorphic tangent bundle which describes the deformation of the complex structure.  Note that it is not necessary to have any information on the complex coordinates of the deformed complex manifold to apply this method of computing $\del_t$ and $\delbar_t$.

We remark that the method just introduced of computing $\del_t$ and $\delbar_t$ acting on $(p,q)$-forms could be applied to find necessary conditions to the existence of differentiable families of deformations with smooth families of special Hermitian metrics other than the SKT ones, for example balanced metrics.

The paper is organized in the following way. In section \ref{preliminaries}, we recall the basic notions and definitions which will be useful later on. In section \ref{deformations}, we give a brief review of the classical deformation theory, following \cite{MK}, and introduce the extension map mentioned above. In section \ref{main}, we recall the expressions of the complex differentials $\del_t$ and $\delbar_t$ acting on $(p,q)$-forms on a differentiable family of complex manifolds $\{M_t\}_t$, as developed by Rao and Zhao in \cite{RZ}, and prove our main result. In section \ref{applications}, we apply Theorem \ref{thm:main} and Corollary \ref{cor:main} providing two examples of 8-dimensional nilmanifolds admitting a left invariant complex structure, more precisely on a family of nilmanifolds introduced in \cite[Section 2.3]{FT11} and  on a quotient of the product of two copies of the real Heisenberg group $\mathbb{H}(3;\R)$ and $\R^2$, presented in \cite[Example 8]{RT12}.

We remark that the case of 8-dimensional nilmanifolds admitting a left invariant complex structure is of particular interest, since, as noted above, existence results for SKT metrics in dimension eight are not as known as in dimension six.

\medskip\medskip
\noindent{\em Acknowledgments.} The authors would like to sincerely thank Adriano Tomassini, both for his support and encouragement, and for many useful discussions and suggestions.
We are also grateful to Daniele Angella, Anna Fino, and Federico Rossi for interesting conversations and helpful comments.

\section{Notations and preliminaries}\label{preliminaries}
Let $(M, J, g,\omega)$ be an Hermitian manifold, with $J\in\End(TM)$ the integrable almost-complex structure on $M$ and $g$ a Riemannian metric on $M$ compatible with $J$. Let $\omega$ be the $(1,1)-$fundamental form associated to $g$ given by $\omega(\cdot,\cdot)=g(J\cdot,\cdot)$.

The metric $g$ is said to be \emph{strong K\"ahler with torsion}, briefly SKT, if
\begin{equation*}
\del\delbar\omega=0,
\end{equation*}
where $d=\del+\delbar$ is the decomposition induced by the complex structure.
%,where $\del=\pi_1\circ d$ and $\delbar=\pi_2\circ d$, with $\pi_1$ and $\pi_2$ the projections of 
%\[
%d(\textstyle\bigwedge^{p,q}(M))\subset\bigwedge^{p+1,q}(M)\oplus\bigwedge^{p,q+1}(M)
%\]
%onto, respectively $\bigwedge^{p+1,q}(M)$ and $\bigwedge^{p,q+1}(M)$.

Let $\pi\colon E\rightarrow M$ be a complex vector bundle of rank $r$ over $(M,J,g,\omega)$, a Hermitian manifold of complex dimension $n$. For every $p,q$, let $\bigwedge^{p,q}(M,E):=\bigwedge^{p,q}(M)\otimes E$ be the bundle of the $(p,q)$-differential forms on $M$ with values in $E$ and let $\mathcal{A}^{p,q}(M,E):=\Gamma(M,\bigwedge^{p,q}(M,E))$ be the space of its global $\mathcal{C}^{\infty}$-sections.

If $h$ is an Hermitian metric $h$ on $E$,  i.e., a smooth Hermitian scalar product on each fibre of $E$, let us identify $h$ as a $\C$-antilinear isomorphism between $E$ and its dual $E^*$ and consider the usual $\C$-antilinear Hodge $\ast$-operator on $(M,J,g,\omega)$ with respect to $g$ (see \cite{Huy04}). Then
\begin{gather*}
{\ast}_E\colon\textstyle\mathcal{A}^{p,q}(M,E)\rightarrow\mathcal{A}^{n-p,n-q}(M,E^{\ast}),\\
{\ast}_E(\varphi\otimes s):=\ast({\varphi})\otimes h(s), \quad\text{for}\,\, \varphi\otimes s\in\textstyle\mathcal{A}^{p,q}(M,E),\nonumber
\end{gather*}
is a $\C$-antilinear isomorphism depending on the metrics $g$ and $h$, such that $\ast_{E^*}\circ{\ast}_E=(-1)^{p+q}$ on $\bigwedge^{p,q}(M)\otimes E$. In particular, $h(\alpha,\beta)\ast 1=\alpha\wedge{\ast}_E(\beta)$, for $\alpha,\beta\in\bigwedge^{p,q}(M,E)$.

An element of $\mathcal{A}^{p,q}(M,E)$ can be locally written as $\beta=\sum \beta_i\otimes s_i$, with $\beta_i\in\mathcal{A}^{p,q}(M)$ and $(s_1\dots,s_r)$ a local trivialization of $E$. Then we can define
\begin{equation}\label{eq:dbarE}
\delbar_E(\beta):=\sum\delbar(\beta_i)\otimes s_i,
\end{equation}
and the Dolbeault cohomology of a holomorphic vector bundle as
\begin{equation*}
H^{p,q}_{\delbar_E}(M,E):=\displaystyle\frac{\Ker(\delbar_E\colon\mathcal{A}^{p,q}(M,E)\rightarrow\mathcal{A}^{p,q+1}(M,E))}{\im(\delbar_E\colon\mathcal{A}^{p,q-1}(M,E)\rightarrow\mathcal{A}^{p,q}(M,E))}.
\end{equation*}
The ${\ast}_E$-operator can be used to define
\begin{equation}\label{eq:dbar*E}
\delbar_E^{\ast}:=-{\ast}_{E^*}\circ\delbar_{E^*}\circ{\ast}_E
\end{equation}
and hence, the Laplace operator and its harmonic forms:
\begin{gather*}
\Delta_E:=\delbar_E^{\ast}\delbar_E+\delbar_E\delbar_E^{\ast}\\
\mathcal{H}^{p,q}(M,E)=\{\,\beta\in\textstyle\mathcal{A}^{p,q}(M,E):\Delta_E(\beta)=0\,\}.
\end{gather*}
Assume that $M$ is compact. If we define the Hermitian product $\llangle\cdot,\cdot\rrangle$ on $\mathcal{A}^{p,q}(M,E)$ as
\begin{equation*}
\llangle\alpha,\beta\rrangle=\int_M h(\alpha,\beta)\ast 1,
\end{equation*}
the operator $\delbar_E^*$ is the adjoint of $\delbar_E$ and the operator $\Delta_E$ is self-adjoint with respect to $\llangle\cdot,\cdot\rrangle$. 
With these notations, the following Hodge decomposition holds
\begin{gather*}
\textstyle\mathcal{A}^{p,q}(M,E)=\delbar_E(\mathcal{A}^{p,q-1}(M,E))\oplus\mathcal{H}^{p,q}(M,E)\oplus\delbar_E^{\ast}(\mathcal{A0}^{p,q+1}(M,E)),
\end{gather*}
and $\mathcal{H}^{p,q}(M,E)$ is finite-dimensional. Also the space $\mathcal{H}^{p,q}(M,E)$ projects bijectively onto $H_{\delbar_E}^{p,q}(M,E)$ which also is finite-dimensional.

In the following, we will denote by simply $\delbar$ the operator $\delbar_{E}$, and by $\mathcal{A}^{p,q}(E)$ the space $\mathcal{A}^{p,q}(M,E)$, when the setting is clear.

We will call the elements of $\A^{0,q}(T^{1,0}M)$ as \emph{$(0,q)$-vector forms}. Let us assume $\phi=\xi\otimes V$ is a $(0,1)$-vector form with $\xi\in\A^{0,1}M$ and $V\in T^{1,0}M$. We define the contraction map as
\begin{gather*}
i_{\phi}\colon\A^{p,q}(E)\rightarrow\A^{p-1,q+1}(E)\\
i_{\phi}(\alpha\otimes s)=\xi\wedge i_{V}(\alpha)\otimes s,
\end{gather*}
where $i_V(\alpha)$ is the usual interior product of a vector field and a $(p,q)$-differential form, and we extend by linearity this definition to any $\phi\in\A^{0,1}(T^{1,0}M)$. 
 Analogously, we define $i_{\c\phi}(\alpha\otimes s)=\c\xi\wedge i_{\c{V}}(\alpha)\otimes s$ for the conjugate $\c\phi=\c\xi\otimes\c{V}$. Define also the contraction
\begin{gather*}
i_{\phi}\colon\Gamma(T^{0,1}M)\rightarrow\Gamma(T^{1,0}M)\\
i_{\phi}W=\xi(W)V,
\end{gather*}
and set $i_{\c\phi}\c{W}=\c\xi(\c{W})\c{V}$.
We will also denote the map $i_{\phi}$ by the symbol $\phi\intprod$.

The cohomology of Bott-Chern of $(M,J)$ is the datum of the spaces
\[
H_{BC}^{p,q}(M)=\frac{\Ker(\del\colon \mathcal{A}^{p,q}(M)\rightarrow\mathcal{A}^{p+1,q}(M))\cap\Ker(\delbar\colon\mathcal{A}^{p,q}(M)\rightarrow\mathcal{A}^{p,q+1}(M))}{\im(\del\delbar\colon\mathcal{A}^{p-1,q-1}(M)\rightarrow\mathcal{A}^{p,q}(M))}.
\] 
We denote by 
\begin{equation*}
\tilde\Delta_{BC}=
\del\delbar\delbar^*\del^*+
\delbar^*\del^*\del\delbar+\del^*\delbar\delbar^*\del+\delbar^*\del\del^*\delbar
+\del^*\del+\delbar^*\delbar
\end{equation*}
the fourth order self-adjoint elliptic operator known as the Bott-Chern Laplacian, where
\begin{equation*}
\del^*=-*\del*,\ \ \ \delbar^*=-*\delbar*,
\end{equation*}
and $*$ is the $\C$-antilinear Hodge operator for a Hermitian metric $g$ on $(M,J)$. We denote by
\[
\mathcal{H}_{BC}^{p,q}(M,g)=\Ker(\tilde\Delta_{BC})\cap\mathcal{A}^{p,q}(M), 
\]
the $(p,q)$-Bott-Chern harmonic forms. If $M$ is compact, by Hodge theory, see \cite[Section 2.b]{S}, we have the following isomorphism of vector spaces
\[
\mathcal{H}_{BC}^{p,q}(M,g)\simeq H_{BC}^{p,q}(M),
\]
induced by the identity map.

\section{Review of deformation theory of complex structures}\label{deformations}
For the sake of completeness, we recall the fundamental definitions and results of deformation theory of complex manifolds both in the differentiable and holomorphic settings which will be useful for our purposes.
Let $B$ be a domain of $\R^m$ (resp. $\C^m$) and $\{M_t\}_{t\in B}$ a family of compact  complex manifolds.

\begin{defi}\label{def:def}
We say that $M_t$ \emph{depends differentiably} (resp. \emph{holomorphically}) on $t\in B$ and that $\{M_t\}_{t\in B}$ \emph{forms a differentiable} (resp. \emph{holomorphic}, or \emph{complex analytic}) \emph{family} if there is a differentiable (resp. complex) manifold $\mathcal{M}$ and a differentiable (resp. holomorphic) proper map $\pi$ from $\mathcal{M}$ onto $B$ such that
\begin{enumerate}
\item $\pi^{-1}(t)=M_t$ as a complex manifold for every $t\in B$,
\item the rank of the Jacobian of $\pi$ is equal to the dimension (resp. complex dimension) of $B$ at each point of $\mathcal{M}$.
\end{enumerate}
We will sometimes denote by $(\mathcal{M},\pi,B)$ the differentiable (resp. complex analytic) family $\{M_{t}\}_{t\in B}$.
\end{defi}
It follows from $(2)$ of the definition that every $M_t$, for $t\in B$, is a submanifold (resp. complex submanifold) of $\mathcal{M}$.
\begin{defi}
If $M$, $N$ are compact  complex manifolds, we say that $M$ \emph{is a differentiable} (resp. \emph{holomorphic}) \emph{deformation of $N$} if there exists a differentiable (resp. holomorphic) family $\{M_t\}_{t\in B}$ over a domain $B$ of $\R^m$ (resp. $\C^m$), with $M_{t_0}=M$, $M_{t_1}=N$ for some $t_0,t_1\in B$.
\end{defi}
A classical theorem by Ehresmann, see \cite{E47} or \cite[Proposition 6.2.2]{Huy04}, shows that if $\{M_t\}_{t\in B}$ is a differentiable family of complex manifolds, then $M_{t_1}$ and $M_{t_2}$ are diffeomorphic as differentiable manifolds for any $t_1,t_2\in B$. Hence, from the differentiable point of view, it holds
\begin{equation}\label{eq:ehr_thm}
\mathcal{M}\simeq M_{t_0}\times B,
\end{equation}
i.e., the manifold $\mathcal{M}$ can be regarded as the product of a fixed $M_{t_0}$, for $t_0\in B$, and the base manifold $B$.

Let $(\mathcal{M},\pi,B)$ be a differentiable family of compact complex manifolds over $B$. For the sake of simplicity we assume $t_0=0$ and $B=B(0,1)\subset\R^m$, i.e. $B=\{t\in\R^m : |t|<1 \}$. 

Let us consider a system of local coordinates $\{\mathcal{U}_j,(\zeta_j,t)\}$ of $\mathcal{M}$ such that each $\mathcal{U}_j$ can be identified with
\begin{equation*}
\{(\zeta_j(p),t(p)): |\zeta_j(p)|<1,|t(p)|<1\}, \quad \pi(\zeta_j(p),t(p))=t(p),
\end{equation*}
with transition functions $f_{jk}$, which identify points in $\mathcal{U}_j\cap\mathcal{U}_k\neq\emptyset$ by
\[
\zeta_k=f_{jk}(\zeta_j,t),
\]
and which are differentiable on $(z,t)$ and are holomorphic on $z$ for any fixed $t$.

By (\ref{eq:ehr_thm}), we can describe local coordinates of $\mathcal{U}_j$ as differentiable functions of coordinates of $M_0=\pi^{-1}(0)$:
\begin{equation}\label{eq:z_zeta}
\zeta_j=\zeta_j(z,t),
\end{equation}
where $z$ are local coordinates on $M_0$. We note that $\zeta_j(z,t)$ is a differentiable function of $(z,t)$, whereas it depends holomorphically on $z$ for a fixed value of $t$.

With the aid of the expressions (\ref{eq:z_zeta}) for the coordinates, we can actually describe the complex structure on each $M_t$, $t\in B$, via a smooth $(0,1)$-vector form $\phi(t)\in\A^{0,1}(T^{1,0}M_0)$, defined starting from the local transition functions $f_{jk}$ (see \cite[page 150]{MK}).

In fact, since both $\{\zeta_j^1(z,0),\dots,\zeta_j^n(z,0)\}$ and $\{z^1,\dots,z^n\}$ are local holomorphic coordinates on $M_0$, where $n=\dim_{\C}M_0$,
\[
\det\left(\frac{\del \zeta_j^{\alpha}(z,0)}{\del z^{\lambda}}\right)_{\alpha}^\lambda\neq 0.
\]
Therefore, in a small neighborhood of $t=0$
\[
\det\left(\frac{\del \zeta_j^{\alpha}(z,t)}{\del z^{\lambda}}\right)_{\alpha}^\lambda\neq 0.
\]
Set $A:=\left(\left(\frac{\del \zeta_j^{\alpha}(z,t)}{\del z^{\lambda}}\right)_{\alpha}^\lambda\right)^{-1}$. Therefore, the local expression
\begin{equation}\label{eq:phi(t)_coord}
\phi(t)=\sum_{\lambda=1}^n\,\phi^{\lambda}\otimes \frac{\del}{\del z^{\lambda}},
\end{equation}
with, for each $\lambda\in\{1,\dots,n\}$,
\begin{equation}\label{eq:phi(t)_coord2}
\phi^{\lambda}=\sum_{\alpha=1}^nA_{\alpha}^\lambda\delbar\zeta_j^{\alpha}\in\mathcal{A}^{0,1}(M_0)
\end{equation}
defines a global $(0,1)$-vector form on $M_0$.

We notice that, by equations \eqref{eq:phi(t)_coord} and \eqref{eq:phi(t)_coord2}, it holds
\[
i_{\phi(t)}\,\,\zeta_j^{\alpha}(z,t)=\sum_{\lambda=1}^n\,\,\phi^{\lambda}\,\frac{\del\zeta_j^{\alpha}}{\del z^{\lambda}}=\delbar\zeta_j^{\alpha}(z,t)
\]
or equivalently
\begin{equation*}
\left(\delbar-\sum_{\lambda=1}^n\phi^{\lambda}\otimes\frac{\del}{\del z^{\lambda}}\right)\zeta_j^{\alpha}(z,t)=0.
\end{equation*}
It can be proved (see \cite[Chapter 4, Proposition 1.2]{MK}) that the (local) holomorphic functions on each $M_t$ are defined as the differentiable functions $f$ defined on open sets of $M_0$ which are solutions to equation
\begin{equation}\label{eq:hol_fun}
\left(\delbar-\sum_{\lambda=1}^n\phi^{\lambda}\otimes\frac{\del}{\del z^{\lambda}}\right)f(z,t)=0,
\end{equation}
i.e., the complex structure on each $M_t$, for $t$ small enough, is encoded in the $(0,1)$-vector form $\phi(t)$.

On the spaces $\A_{q}:=\mathcal{A}^{0,q}(T^{1,0}M_0)$, $q\in\{1,\dots,n\}$, a bracket can be defined in the following way. Let $\Psi=\sum\psi^{\alpha}\del_{\alpha}$ and $\Xi=\sum\xi^{\alpha}\del_{\alpha}$ be respectively $(0,p)$- and a $(0,q)$-vector forms, where $\del_{\alpha}=\de{}{z^\alpha}$. Then
\begin{equation}\label{eq:bracket}
[\,\Psi\,,\,\Xi\,]:=\sum_{\alpha,\beta=1}^n\big(\psi^{\alpha}\wedge\del_{\alpha}\xi^{\beta}-(-1)^{pq}\xi^{\alpha}\wedge\del_{\alpha}\psi^{\beta}\big)\del_{\beta}\quad\in\,\A_{p+q}.
\end{equation}
In particular $[\,,\,]$ is bilinear and satisfies the following
\begin{enumerate}
\item $[\Psi,\Xi]=-(-1)^{pq}[\Xi,\Psi]$,
\item $\delbar[\Psi,\Xi]=[\delbar\Psi,\Xi]+(-1)^p[\Psi,\delbar\Xi]$,
\item $(-1)^{pr}[\Psi[\Xi,\Phi]]+(-1)^{qp}[\Xi,[\Phi,\Psi]]+(-1)^{rq}[\Phi,[\Psi,\Xi]]$=0,
\end{enumerate}
if $\Psi\in\A_p$, $\Xi\in\A_q$ and $\Phi\in\A_r$.

A classical results (see \cite[Chapter 4, Theorem 1.1]{MK}) shows that the deformations of the complex structure on a compact complex manifold can be characterized according to the following theorem.
\begin{thm}\label{thm:Kod-psi}
If $(\mathcal{M},\pi,B)$ is a differentiable family of compact complex manifolds, then the complex structure on each $M_t=\pi^{-1}(t)$ is represented by the vector $(0,1)$-form $\phi(t)\in\A_1$ just constructed on $M_0$, such that $\phi(0)=0$ and
\begin{equation}\label{eq:MC-eq}
\delbar\phi(t)-\frac{1}{2}[\phi(t),\phi(t)]=0\qquad\textit{(Maurer-Cartan equation).}
\end{equation}
%\item $\frac{\del M_t}{\del t}|_{t=0}$ corresponds to $\eta=-\frac{\del\Psi(t)}{\del t}|_{t=0}$.
\end{thm}

As for the existence of deformations of compact complex manifolds, we refer to the general theory known as \emph{Kuranishi theory}.

\begin{comment}
First, let us a define the notion of complete family of deformations.
\begin{defi}\label{def:complete_fam}
If $M$ is a complex manifold, $\pi\colon\mathcal{M}\rightarrow B$ a complex analytic family over $B$, we say that $(\mathcal{M},B,\pi)$ is \emph{complete} at $b\in B$ if for any family $(\mathcal{N},A,\rho)$ such that $\rho^{-1}(a)=\pi^{-1}(b)=M_b$ there is a nieghborhood $U\ni a$ and holomorphic maps $\Phi\colon\rho^{-1}(U)\rightarrow\mathcal{M}$, $h\colon U\rightarrow B$ such that
\begin{enumerate}
\item the following diagram
\begin{equation*}
\xymatrix{ 
\rho^{-1}(U) \ar[r] \ar[d]_{\rho} &   \mathcal{M} \ar[d]^{\pi}\\
 U \ar[r]^{h} & B}
\end{equation*}
commutes;
\item $\Phi$ maps $\rho^{-1}(s)$ biholomorphically onto $\pi^{-1}(h(s))$ for each $s\in U$;
\item $\Phi\colon\pi^{-1}(b)=M_b\rightarrow M_b$ is the identity map.
\end{enumerate}
\end{defi}
\end{comment}

Let $M$ be a compact complex manifold. Fix an Hermitian metric $h$ on $M$, extend it to $\mathcal{A}_q$ and denote it by the same symbol $h$. Define and inner product on $\mathcal{A}_q$ by
\[
\llangle\Psi,\Xi\rrangle=\int_M h(\Psi,\Xi)*1,
\]
where $\Psi,\Xi\in\mathcal{A}_q$, $\ast$ is the $\C$-antilinear Hodge operator. We also define the Laplacian on $\mathcal{A}_q$ by
\[
\square=\delbar^*\delbar+\delbar\delbar^*,
\]
where $\delbar^*$ is the adjoint operator of $\delbar$ with respect to the Hermitian metric $h$. The space of harmonic forms is
\[
\mathcal{H}^q=\{\Psi\in\mathcal{A}_q:\square\Psi=0\}.
\]
The Hodge theory induces a decomposition on the space $\mathcal{A}_q$ as a direct sum of orthogonal subspaces:
\[
\mathcal{A}_q=\mathcal{H}^q\oplus\square\mathcal{A}_q.
\]
The operator $G\colon\mathcal{A}_q\rightarrow
\square\mathcal{A}_q$ is well defined and acts on $\mathcal{A}_q$ as the projection onto $\square\mathcal{A}_q$, whereas the operator $H$ is the well-defined projection operator onto $\mathcal{H}^q$.
\begin{thm}[Kuranishi]\label{thm-Kur}
Let $M$ be a compact complex manifold, $\{\eta_{\nu}\}$ a basis for $\mathcal{H}^1$. Let $\phi(t)$ be the $(0,1)$-vector form which is a power series solution of the equation
\begin{equation}\label{eq:thm-Kur}
\phi(t)=\eta(t)+\frac{1}{2}\delbar^* G[\phi(t),\phi(t)],
\end{equation}
where $\eta(t)=\sum_{\nu=1}^m t_{\nu}\eta_{\nu}$, $|t|<r$, $r>0$, and let $S=\{t\in B_r(0): H[\phi(t),\phi(t)]=0\}$. 
Then for each $t\in S$, $\phi(t)$ determines a complex structure $M_t$ on $M$.
\end{thm}
The space $S$ is called the \emph{space of Kuranishi}.
The proof of Theorem \ref{thm-Kur} shows that a $(0,1)$-vector form $\phi(t)$ satisfying equation (\ref{eq:thm-Kur}) can be constructed as a converging power series
\[
\phi(t)=\sum_{\mu=1}^{\infty}\phi_{\mu}(t)
\]
in which the forms
\[
\phi_{\mu}(t)=\sum_{\nu_{1}+\dots+\nu_{m}=\mu}\phi_{\nu_1\dots\nu_m}t_1^{\nu_1}\cdots t_m^{\nu_m},\quad \phi_{\nu_1\dots\nu_m}\in\mathcal{A}_1,
\]
are determined via a recursive formula. In fact, if $\{\eta_{\nu}\}_{\nu=1}^n$ is a basis for $\mathcal{H}^1$ and we set $\psi_1(t)=\sum_{\nu=1}^mt_{\nu}\eta_{\nu}$, 
equation (\ref{eq:thm-Kur}) assures that each term $\phi_{\mu}$ can be computed as
\begin{equation}\label{eq:rec-Kur}
\phi_{\mu}(t)=\frac{1}{2}\delbar^* G\,\Big(\,\sum_{\kappa=1}^{\mu-1}\,\,[\phi_{\kappa}(t),\phi_{\mu-\kappa}(t)]\,\Big).
\end{equation}

In general $S$ can have singularities and hence may not have a structure of smooth manifold. Nonetheless, $\{M_t\}_{t\in S}$ can be proven to be a locally complete family of complex manifolds and therefore can be still be interpreted as a complex analytic family, see \cite{Kur65}.

As a first step to understand deformations, it makes sense to study how the decompositions of the complexified cotangent bundle $(T_{\C}M)^*$ and its powers $\bigwedge_{\C}^k(M)$ vary along with $M_t$, for a differentiable family $(\mathcal{M},\pi,B)$. For simplicity, we suppose that $B=I=(-\epsilon,\epsilon)\subset\R$, for $\epsilon>0$.

Let us denote the central fiber $M_0=\pi^{-1}(0)$ by $M$ and let us suppose $\phi(t)\in\A_1$ is the $(0,1)$-vector form describing $(\mathcal{M},\pi,B)$. If we denote by $i_{\phi(t)}^k:=\underbrace{i_{\phi(t)}\circ\dots\circ i_{\phi(t)}}_{k\,\,\text{times}}$ and $\overline{\phi(t)}\in\A^{1,0}(T^{0,1}M)$ the conjugate of $\phi(t)$, in the following operators
\begin{equation*}
e^{i_{\phi(t)}}=\sum_{k=0}^{\infty}\frac{1}{k!}i_{\phi(t)}^k \qquad\text{and}\qquad e^{i_{\c{\phi(t)}}}=\sum_{k=0}^{\infty}\frac{1}{k!}i_{\overline{\phi(t)}}^k
\end{equation*}
the summations are finite, since the dimension of $M$ is finite. As in \cite[Definition 2.8]{RZ}, we define the \emph{extension map}
\begin{align}\label{eq:isom_0t}
e^{i_{\phi(t)}|i_{\overline{\phi(t)}}}\colon\A^{p,q}(M)&\rightarrow\A^{p,q}(M_t),
\end{align}
where, if $\alpha=\alpha_{i_1\dots i_p j_1\dots j_q}dz^{i_1}\wedge\dots\wedge dz^{i_p}\wedge d\overline{z}^{j_1}\wedge\dots\wedge d\overline{z}^{j_q}$ is a $(p,q)$-differential form on $M$ with $\alpha_{i_1\dots i_p j_1\dots j_q}$ differentiable functions on $M$ with complex values, we set
\begin{equation}\label{def-exp}
e^{i_{\phi(t)}|i_{\overline{\phi(t)}}}(\alpha)=\alpha_{i_1\dots i_p j_1\dots j_q}e^{i_{\phi(t)}}(dz^{i_1}\wedge\dots\wedge dz^{i_p})\wedge e^{i_{\overline{\phi(t)}}}(d\overline{z}^{j_1}\wedge\dots\wedge d\overline{z}^{j_q}).
\end{equation}
Indeed, we have the following lemma, see \cite[Lemma 2.9, 2.10]{RZ}.
\begin{lem}\label{thm-exp}
For any $p,q$ and for $t$ small, the map $e^{i_{\phi(t)}|i_{\overline{\phi(t)}}}\colon\A^{p,q}(M)\rightarrow\A^{p,q}(M_t)$ is a real linear isomorphism.
\end{lem}
Moreover, the following decompositions hold
\begin{equation}\label{eq:decomp_t}
\A_{\C}^k(M)=\oplus_{p+q=k}\A^{p,q}(M_t),\qquad k\in\{1,\dots,n\}.
\end{equation}

\begin{rmk}\label{rmk-int}
We observe that, for a $(0,1)$-vector form $\phi(t)\in\A_1$ on $M_0$ such that $\phi(0)=0$, the Maurer-Cartan equation (\ref{eq:MC-eq}) is equivalent to the integrability of the complex structure $J_t$ on $M_t$, i.e.,
\begin{equation}\label{eq:integr}
(d\alpha)^{0,2}=0 \qquad \forall\alpha\in\A^{1,0}(M_t),
\end{equation}
where $(d\alpha)^{0,2}$ is the component in $\mathcal{A}^{0,2}(M_t)$ of the $2$-form $d\alpha$, according to decomposition (\ref{eq:decomp_t}). Indeed, from Lemma \ref{thm-exp} it immediately follows $(I-\phi)\intprod:\Gamma(T^{1,0}M)\to \Gamma(T^{1,0}M_t)$ is an isomorphism for $t$ small, and  for $X,Y\in \Gamma(T^{1,0}M)$
\begin{equation*}
-d(\alpha+e^{i_{\phi(t)}|i_{\overline{\phi(t)}}}(\alpha))(X-\phi(t)(X),Y-\phi(t)(Y))=\alpha\left((\delbar\phi(t)-\frac{1}{2}[\phi(t),\phi(t)])(X,Y)\right).
\end{equation*}
See also \cite[Proposition 6.1.2]{Huy04}.
Furthermore, for a $(0,1)$-vector form satisfying (\ref{eq:thm-Kur}), the defining property of $S$, i.e., $H[\phi(t),\phi(t)]=0$, is equivalent to the integrability condition given by the Maurer-Cartan equation (\ref{eq:MC-eq}) (see \cite[Chapter 4, Proposition 2.5]{MK}).
\end{rmk}

\section{Proof of Theorem \ref{thm:main}}\label{main}
%In this section we prove our main result. 
Let $(\mathcal{M},\pi,I)$ be a differentiable family of compact complex manifolds parametrized by $\phi(t)$, for $t\in I$, $I=(-\epsilon,\epsilon)$, $\epsilon>0$.
We need to recall formulas for the differential operators $\del_t$ and $\delbar_t$, defined as
\begin{gather*}
\del_t:=\pi_t^{p+1,q}\circ d\colon \A^{p,q}(M_t)\rightarrow\A^{p+1,q}(M_t),\\
\delbar_t:=\pi_t^{p,q+1}\circ d\colon \A^{p,q}(M_t)\rightarrow\A^{p,q+1}(M_t),
\end{gather*}
for any $p,q$, with $\pi_t^{p+1,q}$ and $\pi_t^{p,q+1}$ the usual projections of $d(\A^{p,q}(M_t))$ with respect to the decompositions (\ref{eq:decomp_t}).

We take as main reference \cite{RZ}.
Starting from $(0,0)$-differential forms, i.e., differentiable complex functions, we have
\begin{gather}
\del_t f=e^{i_{\phi}}\Big((I-\phi\overline{\phi})^{-1}\intprod(\del-\overline{\phi}\intprod\delbar)f\Big),\\
\delbar_t f=e^{i_{\overline{\phi}}}\Big((I-\overline{\phi}\phi)^{-1}\intprod(\delbar-\phi\intprod\del)f \Big),
\end{gather}
where $\phi\overline{\phi}=\overline{\phi}\intprod\phi$, $\overline{\phi}\phi=\phi\intprod\overline{\phi}$ and we omit the dependence on $t$ of $\phi$, see \cite[Equation (2.13)]{RZ}.
We will denote by $\Finv$ the simultaneous contraction on each component of complex differential form, i.e.
\begin{align*}
\psi\Finv\alpha:=
\alpha_{i_1\dots i_p j_1\dots j_q}\psi\intprod dz^{i_1}\wedge \dots\wedge \psi\intprod dz^{i_p}\wedge \psi\intprod d\overline{z}^{j_1}\wedge \dots \wedge \psi\intprod d\overline{z}^{j_q},
\end{align*}
for $\psi\in\A_1+\c{\A_1}$ and for any $(p,q)$-form locally written as $\alpha=\alpha_{i_1\dots i_p j_1\dots j_q}dz^{i_1}\wedge\dots\wedge dz^{i_p}\wedge d\overline{z}^{j_1}\wedge\dots\wedge d\overline{z}^{j_q}$. This contraction is well-defined and it can be used to describe the extension map, in fact
\begin{equation*}
e^{i_{\phi(t)}|i_{\overline{\phi(t)}}}=(I+\phi+\overline{\phi})\Finv.
\end{equation*}
%Note that $\Finv$ is not a linear operator over $(p,q)$-forms.
With these notations, from the proof of \cite[Proposition 2.13]{RZ}, we can summarize the action of the operators $\del_t$ and $\delbar_t$ on differential forms $e^{i_{\phi(t)}|i_{\overline{\phi(t)}}}\alpha\in\A^{p,q}(M_t)$, with $\alpha\in\A^{p,q}(M)$. Then,
\begin{align}
\del_t(e^{i_{\phi}|i_{\overline{\phi}}}\alpha)&=e^{i_{\phi}|i_{\overline{\phi}}}\Big((I-\phi\overline{\phi})^{-1}\Finv([\delbar,i_{\overline{\phi}}]+\del)(I-\phi\overline{\phi})\Finv \alpha\Big)%\nonumber\\
%&=(I+\phi+\overline{\phi})\Finv\Big((I-\phi\overline{\phi})^{-1}\Finv([\delbar,i_{\overline{\phi}}]+\del)(I-\phi\overline{\phi})\Finv \alpha\Big)
,\label{delt}\\
\delbar_t(e^{i_{\phi}|i_{\overline{\phi}}}\alpha)&=e^{i_{\phi}|i_{\overline{\phi}}}\Big((I-\overline{\phi}\phi)^{-1}\Finv([\del,i_{\phi}]+\delbar)(I-\overline{\phi}\phi)\Finv\alpha\Big)%\nonumber\\
%&=(I+\phi+\overline{\phi})\Finv\Big((I-\overline{\phi}\phi)^{-1}\Finv([\del,i_{\phi}]+\delbar)(I-\overline{\phi}\phi)\Finv\alpha\Big)
.\label{delbart}
\end{align}

Now we have all the ingredients to prove our main result Theorem \ref{thm:main}. Let us fix $(M,J,g,\omega)$ a compact Hermitian manifold and suppose that $g$ is SKT, i.e. $\del\delbar\omega=0$. We want to find necessary conditions under which the property of being SKT is stable for a smooth family of Hermitian metrics $\{\omega_t\}_{t\in I}$ such that $\omega_0=\omega$, along a deformation of the complex structure parametrized by a $(0,1)$-vector form $\phi(t)$. %Suppose that each $\omega_t$ is SKT for any $t\in(-\epsilon,\epsilon)=I$, i.e. $\del_t\delbar_t\omega_t=0$. Using the Taylor series expansion and differentiating this expression with respect to $t$, we obtain the following.
\begin{comment}
\begin{thm}\label{thm:main}
Let $(M,J,g,\omega)$ be a compact Hermitian manifold with $g$ a SKT metric. Let $\{M_t\}_{t\in I}$ be a differentiable family of compact complex manifolds parametrized by $\phi(t)\in\A_1$, for $t\in I=(-\epsilon,\epsilon)$, $\epsilon>0$. Let $\{\omega_t\}_{t\in I}$ be a smooth family of Hermitian metrics on each $M_t$ written as
\begin{equation}
\omega_t=e^{i_{\phi(t)}|i_{\overline{\phi(t)}}}\,\,(\omega(t)),
\end{equation}
where $\omega(t)$ has local expression $\omega_{ij}(t)\, dz^i\wedge d\overline{z}^j\in\A^{1,1}(M)$. Denote by $\omega'(t):=\de{}{t}\omega_{ij}(t)\, dz^i\wedge d\overline{z}^j\in\A^{1,1}(M)$. Then, if the metrics $\omega_t$ are SKT for every $t\in I$, the following condition must hold
\begin{equation}\label{eq:main}
2i\mathfrak{Im}(\del\circ i_{\phi'(0)}\circ \del)(\omega)=\del\delbar\omega'(0).
\end{equation}
\end{thm}
As a consequence, we have the following cohomological obstruction.
\begin{cor}\label{cor:main}
Let $(M,J,g,\omega)$ be a compact Hermitian manifold. A necessary condition for the existence of a smooth family of SKT metrics which equals $\omega$ in $t=0$ along the family of deformations $t\mapsto\phi(t)$ is that the following equation must hold
\[
\left[\mathfrak{Im}(\del\circ i_{\phi'(0)}\circ \del)(\omega)\right]_{H_{BC}^{2,2}(M)}=0.
\]
\end{cor}
\end{comment}

\begin{proof}[Proof of Theorem \ref{thm:main}]
The metrics $\omega_t$ are SKT for every $t\in I$, i.e., $\del_t\delbar_t\omega_t=0$. This implies
\begin{equation}\label{eq:deldelbart}
\de{}{t}(\del_t\delbar_t\omega_t)_{|t=0}=0.
\end{equation}
Let us compute equation (\ref{eq:deldelbart}) using the expressions (\ref{delt}) and (\ref{delbart}) for $\del_t$ and $\delbar_t$. First we calculate $\delbar_t(\omega_t)$
\begin{align*}
\delbar_t(\omega_t)&=e^{i_{\phi}|i_{\overline{\phi}}}\Big((I-\overline{\phi}\phi)^{-1}\Finv([\del,i_{\phi}]+\delbar)(I-\overline{\phi}\phi)\Finv\omega(t)\Big),
\end{align*}
and then $\del_t\delbar_t(\omega_t)$,
\begin{align*}
\del_t\delbar_t(\omega_t)&=e^{i_{\phi}|i_{\overline{\phi}}}\Big((I-\phi\overline{\phi})^{-1}\Finv([\delbar,i_{\overline{\phi}}]+\del)(I-\phi\overline{\phi})\Finv (I-\overline{\phi}\phi)^{-1}\Finv([\del,i_{\phi}]+\delbar)(I-\overline{\phi}\phi)\Finv\omega(t)\Big).
\end{align*}
Now, to compute equation (\ref{eq:deldelbart}), we develop $\del_t\delbar_t(\omega_t)$ in Taylor series centered in $t=0$ up to the first order. Note that
\begin{equation*}
\phi(t)=t\phi'(0)+o(t)
\end{equation*}
implies
\begin{equation*}
(I-\phi\overline{\phi})=(I-\overline{\phi}\phi)=(I-\phi\overline{\phi})^{-1}=(I-\overline{\phi}\phi)^{-1}=I+o(t).
\end{equation*}
Therefore we get
\begin{align*}
\del_t\delbar_t(\omega_t)&=(I+t\phi'(0)+t\overline{\phi'(0)})\Finv([\delbar,{t\overline{\phi'(0)}}\intprod]+\del) ([\del,{t\phi'}(0)\intprod]+\delbar)\Big(\omega(0)+t\omega'(0)\Big)+o(t)\\
&=(I+t\phi'(0)+t\overline{\phi'(0)})\Finv([\delbar,{t\overline{\phi'(0)}}\intprod]+\del)\Big([\del,{t\phi'}(0)\intprod]\omega(0)+\delbar\omega(0)+t\delbar\omega'(0)\Big)+o(t)\\
&=(I+t\phi'(0)+t\overline{\phi'(0)})\Finv\Big(-t\del({\phi'}(0)\intprod\del\omega(0))+t\delbar({\overline{\phi'(0)}}\intprod\delbar\omega(0))+t\del\delbar\omega'(0)\Big)+o(t)\\
&=-t\del({\phi'}(0)\intprod\del\omega(0))+t\delbar({\overline{\phi'(0)}}\intprod\delbar\omega(0))+t\del\delbar\omega'(0)+o(t),
\end{align*}
implying
\begin{equation*}
0=\de{}{t}(\del_t\delbar_t\omega_t)_{|t=0}=-\del({\phi'}(0)\intprod\del\omega(0))+\delbar({\overline{\phi'(0)}}\intprod\delbar\omega(0))+\del\delbar\omega'(0),
\end{equation*}
which is equivalent to equation \eqref{eq:main}.
\end{proof}

\begin{comment}
\begin{rmk}
From the proof of Theorem \ref{thm:main}, it is clear that the same computations can be carried out in the case the differentiable family of complex manifolds is $n$-dimensional, i.e., the parameter $t$ lies on a open set $B\subset\R^n$, with $1<n\in\N$. However, the more information is obtained when applying the Theorem in the case $n=1$, where one can choose a $1$-dimensional direction to investigate if there exists a family of SKT metrics on it.
\end{rmk}
\end{comment}

\section{Applications}\label{applications}
We now apply Corollary \ref{cor:main} and Theorem \ref{thm:main} to study two $4$-dimensional complex nilmanifolds admitting invariant SKT metrics. In particular, we study obstructions along a specific family of deformations on a family of nilmanifolds introduced in \cite[Section 2.3]{FT11} and  on a quotient of the product of two copies of the real Heisenberg group $\mathbb{H}(3;\R)$ and $\R^2$ presented in \cite[Example 8]{RT12}. 

In the following, we may refer to one-dimensional differentiable families of complex manifolds $\{M_t\}_{t\in I}$, $I=(-\epsilon,\epsilon)$, $\epsilon>0$, by the terminology \emph{curves of complex structures}.

\subsection{Example 1}
Let us consider the Lie algebra $\mathfrak{g}$ endowed with integrable almost complex structure $J$ such that $\mathfrak{g}^*$ is spanned by $\{\eta^1,\dots,\eta^4\}$, a set of $(1,0)$ complex differential forms with structure equations
\begin{equation}\label{eq:RT_struct_eq_0}
\begin{cases}
\,\,d\eta^i&=0,\qquad i\in\{1,2,3\},\\
\,\,d\eta^4&=a_1\eta^{12}+a_2\eta^{13}+a_3\eta^{1\overline{1}}+a_4\eta^{1\overline{2}}+a_5\eta^{1\overline{3}}\\
&+a_6\eta^{23}+a_7\eta^{2\overline{1}}
+a_8\eta^{2\overline{2}}+a_9\eta^{2\overline{3}}\\
&+a_{10}\eta^{3\overline{1}}+a_{11}\eta^{3\overline{2}}+a_{12}\eta^{3\overline{3}},
\end{cases}
\end{equation}
with $a_i\in\C$ for $i\in\{1,\dots,12\}$. In particular, $\mathfrak{g}$ is a $2$-step nilpotent Lie algebra depending on the complex parameters $a_1,\dots,a_{12}$. If we denote by $G$ the simply-connected nilpotent Lie group with Lie algebra $\mathfrak{g}$, then for any $a_1,\dots,a_{12}\in\Q[i]$, by Malcev's theorem \cite[Theorem 7]{M62}, there exists a uniform discrete subgroup $\Gamma$ of $G$ such that $M=\Gamma/G$ is a nilmanifold. As in \cite[Theorem 2.7]{FT11}, the invariant Hermitian metric on $M$
\[
g=\frac{1}{2}\sum_{j=1}^4 (\eta^j\otimes\overline{\eta}^j + \overline{\eta}^j\otimes \eta^j)
\]
is Astheno K\"ahler, i.e., the fundamental form of $g$
\begin{equation}\label{eq:RT_omega0}
\omega=\frac{i}{2}\sum_{j=1}^4\eta^j\wedge\overline{\eta}^j
\end{equation}
is such that $\del\delbar\omega^2=0$, if and only if the following equation holds
\begin{equation}\label{eq:RT_AK}
|a_1|^2+|a_2|^2+|a_5|^2+|a_6|^2+|a_7|^2+|a_9|^2+|a_{10}|^2+|a_{11}|^2=2\mathfrak{Re}(a_3\overline{a}_8+a_3\overline{a}_{12}+a_8\overline{a}_{12}).
\end{equation}
Moreover, if $a_8=0$, the Astheno-K\"ahler metric $g$ is SKT if and only if
\begin{equation*}
a_1=a_4=a_6=a_7=a_9=a_{11}=0.
\end{equation*}
Hence, if $a_i=0$ for $i\in\{1,4,6,7,8,9,11\}$  and
\begin{equation}\label{eq:RT_AK-SKT}
|a_2|^2+|a_5|^2+|a_{10}|^2=2\mathfrak{Re}(a_3\overline{a}_{12}),
\end{equation}
from equation (\ref{eq:RT_AK}), the metric $g$ is SKT, i.e., $\del\delbar\omega=0$. From now on, we will consider the nilmanifold $(M,J)$, with Hermitian SKT metric $\omega$.

The structure equations (\ref{eq:RT_struct_eq_0})  boil down to
\begin{align}\label{eq:RTstruct_eq_1}
\begin{cases}
\,\,d\eta^i&=0,\qquad i\in\{1,2,3\},\\
\,\,d\eta^4&=a_2\eta^{13}+a_3\eta^{1\overline{1}}+a_5\eta^{1\overline{3}}+a_{10}\eta^{3\overline{1}}+a_{12}\eta^{3\overline{3}}.
\end{cases}
\end{align}

We consider now the following invariant $(0,1)$-vector form given by
\begin{equation}
\phi(r,s)=r\overline{\eta}^1\otimes Z_1+ s\overline{\eta}^3\otimes Z_3, \qquad (r,s)\in\C^2,\ |r|<1,\  |s|<1,
\end{equation}
where $Z_j$ is the dual of $\eta^j$ in $\mathfrak{g}$, for $j\in\{1,2,3,4\}$. We define the invariant forms $\eta_{r,s}^j:=\eta^j+i_{\phi}(\eta^j)$, for $j\in\{1,2,3,4\}$:
\begin{align*}
\begin{cases}
\,\eta_{r,s}^1&=\eta^1+r\overline{\eta}^1,\\
\,\eta_{r,s}^2&=\eta^2,\\
\,\eta_{r,s}^3&=\eta^3+s\overline{\eta}^3,\\
\,\eta_{r,s}^4&=\eta^4,
\end{cases}
\end{align*}
which form a coframe of $(T^{1,0}M_t)^*$.
It is clear that
\begin{align*}
\begin{cases}
\,\eta^1&=\frac{1}{1-|r|^2}(\eta_{r,s}^1-r\overline{\eta}_{r,s}^1),\\
\,\eta^2&=\eta_{r,s}^2,\\
\,\eta^3&=\frac{1}{1-|s|^2}(\eta_{r,s}^3-s\overline{\eta}_{r,s}^3),\\
\,\eta^4&=\eta_{r,s}^4.
\end{cases}
\end{align*}
Therefore, it can be easily seen that the structure equations for the coframe $\{\eta_{r,s}^1,\eta_{r,s}^2,\eta_{r,s}^3,\eta_{r,s}^4\}$ are:
\begin{align*}\label{eq:RT_struct_def}
d\eta_{r,s}^i&=0, \qquad i\in\{1,2,3\},\\
d\eta_{r,s}^4&=\frac{a_2+\overline{r}a_{10}-\overline{s}a_5}{(1-|r|^2)(1-|s|^2)}\,\,\eta_{r,s}^{13}\,+\,\frac{a_3}{1-|r|^2}\,\,\eta_{r,s}^{1\overline{1}}
\,+\,\frac{a_5-sa_2-\overline{r}sa_{10}}{(1-|r|^2)(1-|s|^2)}\,\,\eta_{r,s}^{1\overline{3}}+\\
&+\frac{a_{10}+ra_2-r\overline{s}a_5}{(1-|r|^2)(1-|s|^2)}\,\,\eta_{r,s}^{3\overline{1}}\,+\,\frac{a_{12}}{1-|s|^2}\,\,\eta_{r,s}^{3\overline{3}}\,+\,\frac{-ra_5+sa_{10}+rsa_2}{(1-|r|^2)(1-|s|^2)}\,\,\eta_{r,s}^{\overline{13}}.
\end{align*}
For the integrability condition $(d\eta_{r,s}^i)^{0,2}=0$, which is equivalent to check the Maurer Cartan equation for $\phi$ by Remark \ref{rmk-int}, we must have that
\begin{equation}\label{eq:RT_space_def}
-ra_5+sa_{10}+rsa_2=0.
\end{equation}
We begin studying this equation by noticing that, if we set $F(r,s)=-ra_5+sa_{10}+rsa_2$, the gradient $\nabla F$ in $(r,s)=(0,0)$ is
\[
\begin{pmatrix}
F_r(0,0)\\
F_s(0,0)
\end{pmatrix}=
\begin{pmatrix}
-a_5\\
a_{10}
\end{pmatrix}.
\]
We distinguish two cases, depending on whether $\nabla F(0,0)=0$ or $\nabla F(0,0)\neq 0$. We observe that in the first case, the solution set, which we will denote by $B$, might not be a smooth manifold, whereas it happens in the latter case.

\subsubsection{Case $(i)$} $\nabla F(0,0)=0$, i.e., $a_{5}=a_{10}=0$.
The solutions of (\ref{eq:RT_space_def}) are
\begin{equation*}
B=\{(r,s)\in\C^2: rsa_2=0, |r|,|s|<\delta\},
\end{equation*}
for $\delta>0$ sufficiently small. The corresponding $(0,1)$-vector form which parametrizes the deformation is $\phi=r\overline{\eta}^1\otimes Z_1 + s\overline{\eta}^3\otimes Z_3$. If we consider the segment $\gamma\colon(-\epsilon,\epsilon)\rightarrow B$, $\gamma(t)=(tu,tv)$ for $(u,v)\in B$, we define the curve of deformations
\[
t\mapsto\phi(t)=tu\,\overline{\eta}^1\otimes Z_1+tv\,\overline{\eta}^3\otimes Z_3.
\]
In this case, $\phi'(0)=u\,\overline{\eta}^1\otimes Z_1 + v\,\overline{\eta}^3\otimes Z_3$. With structure equations
\begin{align*}
\begin{cases}
d\eta^i&=0, \quad i\in\{1,2,3\},\\
d\eta^4&=a_2\eta^{13}+a_3\eta^{1\overline{1}}+a_{12}\eta^{3\overline{3}}\nonumber,
\end{cases}
\end{align*}
we compute $\del\circ i_{\phi'(0)}\circ \del (\omega)$. It turns out that this term vanishes, therefore Corollary \ref{cor:main} gives no obstructions to the existence of curve of SKT metrics along the curve of deformations $t\mapsto \phi(t)$.

\subsubsection{Case $(ii)$} $\nabla F(0,0)\neq 0$, i.e., $(a_5,a_{10})\neq (0,0)$.

We begin by studying the case $a_5\neq 0.$
The set
\[
B=\left\{(r,s)\in\C^2:r=\frac{s a_{10}}{a_5-s a_2},\ |r|<\delta,|s|<\delta'\right\},
\]
for $\delta,\delta'>0$ sufficiently small, is the set of the solutions of equation (\ref{eq:RT_space_def})
\begin{equation*}
- r a_5+s a_{10}+r s a_2=0.
\end{equation*}

If we consider the smooth curve $\gamma\colon(-\epsilon,\epsilon)\rightarrow B$,
\begin{equation}\label{eq:RT_def_curve_1}
\gamma(t)=(\frac{t u a_{10}}{a_5-t u a_2},t u)
\end{equation}
with $u\in\C$, we have that
\[
t\mapsto \phi(t)=\frac{t ua_{10}}{a_5-tu a_2}\,\overline{\eta}^1\otimes Z_1 +t u\overline{\eta}^3\otimes Z_3
\]
is a smooth curve of deformations with $\phi'(0)=\frac{u a_{10}}{a_{5}}\,\overline{\eta}^1\otimes Z_1+ u\,\overline{\eta}^3\otimes Z_3$. By the usual computations and structure equations
\begin{align*}
\begin{cases}
d\eta^i&=0, \quad i\in\{1,2,3\},\\
d\eta^4&=a_2\eta^{13}+a_3\eta^{1\overline{1}}+a_5\eta^{1\overline{3}}+a_{10}\eta^{3\overline{1}}+a_{12}\eta^{3\overline{3}},
\end{cases}
\end{align*}
we obtain that
\[
\del\circ i_{\phi'(0)}\circ \del (\omega)=iua_2\frac{|a_{10}|^2-|a_5|^2}{a_5} \eta^{13\overline{13}}.
\]
We observe that the real form $\eta^{13\overline{13}}$ is closed with respect to $\del$ and $\delbar$. Moreover,
\[
(\del\delbar \ast) \eta^{13\overline{13}}=(|a_2|^2+|a_3|^2+|a_{10}|^2-2\Re(a_3\overline{a}_{12}))\eta^{123\overline{123}}=0,
\]
by equation (\ref{eq:RT_AK-SKT}). Therefore $\eta^{13\overline{13}}$ is harmonic with respect to the Bott-Chern Laplacian and, via the canonical isomorphism, the class $[\eta^{13\overline{13}}]_{BC}$ is a non-vanishing class in $ H_{BC}^{2,2}(M)$. Hence, if
\[
\mathfrak{Im}\left(iua_2\frac{|a_{10}|^2-|a_5|^2}{a_5}\right)\neq 0,
\]
by Corollary (\ref{cor:main}) there exist no family of SKT metrics $\omega_t$ along $t\mapsto \phi(t)$ such that $\omega_0=\omega$.

If instead we assume that $a_{10}\neq 0$, we have that equation (\ref{eq:RT_space_def})
\[
-r a_5+s a_{10}+r s a_2=0 
\]
admits solutions
\begin{equation*}
B=\left\{(r,s)\in\C^2:\, s=\frac{r a_{5}}{a_{10}+r a_2},\  |r|<\delta,|s|<\delta'\right\},
\end{equation*}
with $\delta,\delta'>0$ sufficiently small.

If $\gamma\colon(-\epsilon,\epsilon)\rightarrow B$ is the smooth curve $\gamma(t)=(t v,\frac{t v a_5}{a_{10}+tv a_2})$ with $v\in\C$, we define the curve of deformations by
\begin{equation}\label{eq:RT_def_curve_2}
t\mapsto \phi(t)=t v\,\overline{\eta}^1\otimes Z_1 + \frac{t v a_5}{a_{10}+t v a_2}\overline{\eta}^3\otimes Z_3.
\end{equation}
We notice that $\phi'(0)=v\,\overline{\eta}^1\otimes Z_1+ v \frac{a_5}{a_{10}}\,\overline{\eta}^3\otimes Z_3$. With the aid of structure equations (\ref{eq:RTstruct_eq_1}), we can check that
\begin{equation*}
\del\circ i_{\phi'(0)}\circ \del (\omega)=iva_2\frac{|a_{10}|^2-|a_5|^2}{a_{10}}\eta^{13\overline{1}\overline{3}}.
\end{equation*}
Since $\eta^{13\overline{1}\overline{3}}\in\mathcal{H}_{BC}^{2,2}(M,g)$ and $[\eta^{13\overline{13}}]_{BC}$ does not represent the class $0\in H_{BC}^{2,2}$, therefore, if
\[
\mathfrak{Im}\left(iva_2\frac{|a_{10}|^2-|a_5|^2}{a_{10}}\right)\neq 0,
\]
by Corollary \ref{cor:main}, there is no curve of SKT metrics $\omega_t$ along the curve of deformations $t\mapsto \phi(t)$ such that $\omega_0=\omega$.

Summing up, we gather what we obtained.
\begin{thm}
Let $(M,J)$ be an element of the familiy of nilmanifolds with structure equations
\begin{align*}
\begin{cases}
\,\,d\eta^i&=0,\qquad i\in\{1,2,3\},\\
\,\,d\eta^4&=a_2\eta^{13}+a_3\eta^{1\overline{1}}+a_5\eta^{1\overline{3}}+a_{10}\eta^{3\overline{1}}+a_{12}\eta^{3\overline{3}},
\end{cases}
\end{align*}
$a_2,a_3,a_5,a_{10},a_{12}\in\mathbb{Q}[i]$ such that $|a_2|^2+|a_5|^2+|a_{10}|^2=2\mathfrak{Re}(a_3\overline{a}_{12})$, endowed with the Hermitian SKT metric $\omega=\frac{i}{2}\sum_{j=1}^4\eta^{j\overline{j}}$. Then
\begin{itemize}
\item if $a_5\neq 0$ and $u\in\C$, there exist no smooth curve of SKT metrics $\omega_t$ such that $\omega_0=\omega$  along the curve of deformation $t\mapsto \phi(t)=\frac{t u a_{10}}{a_5-t u a_2}\,\overline{\eta}^1\otimes Z_1 +t u\overline{\eta}^3\otimes Z_3$ for $t\in(-\epsilon,\epsilon),\epsilon>0$, if
\[
\mathfrak{Im}\left(iua_2\frac{|a_{10}|^2-|a_5|^2}{a_5}\right)\neq 0;
\]
\item if $a_{10}\neq 0$ and $v\in\C$, there exist no smooth curve of SKT metrics $\omega_t$ such that $\omega_0=\omega$  along the curve of deformation $t\mapsto \phi(t)=t v\,\overline{\eta}^1\otimes Z_1 + \frac{t v a_5}{a_{10}+tv a_2}\overline{\eta}^3\otimes Z_3$ for $t\in(-\epsilon,\epsilon),\epsilon>0$, if
\[
\mathfrak{Im}\left(iva_2\frac{|a_{10}|^2-|a_5|^2}{a_{10}}\right)\neq 0.
\]
\end{itemize}
\end{thm}

\subsection{Example 2}
Let us consider the group $G:=\mathbb{H}(3;\R)\times\mathbb{H}(3;\R)\times\R^2$, where $\mathbb{H}(3;\R)$ is the $3$-dimensional real Heisenberg group. We fix a basis $\{e^1,\dots,e^8\}$ for $\mathfrak{g}^*$, the dual of the Lie algebra $\mathfrak{g}$ of $G$ such that
\begin{align*}
\begin{cases}
de^1=de^2=de^3=de^4=de^5=de^7=0,\\
de^6=-\frac{1}{2}e^{12},\quad de^8=-\frac{1}{2}e^{34}.
\end{cases}
\end{align*}
Due to \cite[Theorem 7]{M62}, there exists a lattice $\Gamma$ of $G$ such that the quotient $M=\Gamma/G$ is a compact manifold. In particular, $M$ is real $8$-dimensional nilmanifold.

If we make use of the standard real coordinates $\{x_1,x_2,x_3\}$ and $\{x_4,x_5,x_6\}$ on the two copies of $\mathbb{H}(3;\R)$ and $\{x_7,x_8\}$ on $\R^2$, the coframe $\{e^1,\dots,e^8\}$ can be written as
\begin{align*}
\begin{cases}
e^1=dx^1,\quad  e^2=dx^2,\quad e^6=dx^3-x^1dx^2,\\
e^3=dx^4,\quad e^4=dx^5, \quad e^8=dx^6-x^4dx^5,\\
e^5=dx^7, \quad e^7=dx^8.
\end{cases}
\end{align*}
Notice that it defines a global left-invariant coframe of differential $1$-forms on $G$, and therefore on $M$.

Let us define an almost-complex structure $J$ on $\mathfrak{g}^*$ by setting the following basis for $(\mathfrak{g}^*)^{1,0}$
\begin{align*}
\begin{cases}
\eta^1:=e^1+ie^2,\quad \eta^2:=e^3+ie^4,\\
\eta^3:=e^5+ie^6, \quad 
\eta^4:=e^7+ie^8.
\end{cases}
\end{align*}
Let $Z_j$ be the dual of $\eta^j$ in $\mathfrak{g}$, for $j\in\{1,2,3,4\}$.
This position gives rise to a left-invariant integrable almost-complex structure on $G$, hence it descends to the quotient $M$. With an abuse of notation we will denote the latter by $J$.

%In real coordinates, the coframe $\{\eta^1,\dots,\eta^4\}$ can be written as
%\begin{align*}
%\begin{cases}
%\eta^1=dx^1+idx^2\\
%\eta^2=dx^4+idx^5\\
%\eta^3=dx^7+i(dx^3-x^1dx^2)\\
%\eta^4=dx^8+i(dx^6-x^4dx^5)
%\end{cases}
%\end{align*}
We find  that the holomorphic coordinates on $M$ which induce $J$ are
\begin{align}\label{hol-coord}
\begin{cases}
z^1=x^1+ix^2,\\
z^2=x^4+ix^5,\\
z^3=x^7+\frac{1}{2}(x^2)^2+i(x^3-x^1x^2),\\
z^4=x^8+\frac{1}{2}(x^5)^2+i(x^6-x^4x^5).
\end{cases}
\end{align}
\begin{comment}
so that, in holomorphic coordinates, the coframe $\{\eta^1,\dots,\eta^4\}$ can be written as
\begin{align*}
\begin{cases}
\eta^1=dz^1,\\
\eta^2=dz^2,\\
\eta^3=dz^3-\frac{z^1-\overline{z^1}}{2}dz^1,\\
\eta^4=dz^4-\frac{z^2-\overline{z^2}}{2}dz^2.
\end{cases}
\end{align*}

Hence, the left-invariant dual frame $\{\eta_1,\eta_2,\eta_3,\eta_4\}$ on $\mathfrak{g}$ has expression in coordinates
\begin{align*}
\begin{cases}
\eta_1=\frac{\del}{\del z^1}+\frac{z^1-\overline{z^1}}{2}\frac{\del}{\del z^3}\\
\eta_2=\frac{\del}{\del z^2}+\frac{z^2-\overline{z^2}}{2}\frac{\del}{\del z^4}\\
\eta_3=\frac{\del}{\del z^3}\\
\eta_4=\frac{\del}{\del z^4}
\end{cases}
\end{align*} 

\end{comment}

We point out that the structure equations for $(M,J)$ are
\begin{align}\label{eq:RT2_struct_eq}
\begin{cases}
d\eta^1=d\eta^2=0,\\
d\eta^3=\frac{1}{2}\eta^{1\overline{1}},\\
d\eta^4=\frac{1}{2}\eta^{2\overline{2}}.
\end{cases}
\end{align}

Let us now consider a generic Hermitian invariant metric $g$ with associated fundamental form
\begin{equation*}
\omega=\frac{i}{2}\sum_{j=1}^4\alpha_{j\overline{j}}\,\,\eta^{j\overline{j}}+\frac{1}{2}\sum_{j<k}\left(\alpha_{j\overline{k}}\,\,\eta^{j\overline{k}}-\overline{\alpha}_{j\overline{k}}\,\,\eta^{k\overline{j}}\right),
\end{equation*}
whose coefficients $\alpha_{i\overline{j}}$ are such that the matrix representing $g$
\begin{equation*}
\begin{pmatrix}
\alpha_{1\overline{1}} & -i\alpha_{1\overline{2}} & -i\alpha_{1\overline{3}} & -i\alpha_{1\overline{4}}\\
i\overline{\alpha}_{1\overline{2}} & \alpha_{2\overline{2}} & -i\alpha_{2\overline{3}} & -i\alpha_{2\overline{4}}\\
i\overline{\alpha}_{1\overline{3}} & i\overline{\alpha}_{2\overline{3}} & \alpha_{3\overline{3}} & -i\alpha_{3\overline{4}}\\
i\overline{\alpha}_{1\overline{4}} & i\overline{\alpha}_{2\overline{4}} & i\overline{\alpha}_{3\overline{4}} & \alpha_{4\overline{4}}
\end{pmatrix}
\end{equation*}
is positive definite.

It is straightforward to check with the aid of (\ref{eq:RT2_struct_eq}), that $g$ is a SKT metric if and only if
\begin{equation*}
\mathfrak{Im}(\alpha_{3\overline{4}})=0.
\end{equation*} 

We construct a $(0,1)$-vector form
\begin{align*}
\phi(\mathbf{t})&=t_{11}\overline{\eta}^{1}\otimes Z_{1}+t_{22}\overline{\eta}^{2}\otimes Z_{2}+t_{32}\overline{\eta}^{2}\otimes Z_{3}+t_{33}\overline{\eta}^{3}\otimes Z_{3}\\
&+t_{34}\overline{\eta}^{4}\otimes Z_{3}+t_{41}\overline{\eta}^{1}\otimes Z_{4}+t_{43}\overline{\eta}^{3}\otimes Z_{4}+t_{44}\overline{\eta}^{4}\otimes Z_{4},
\end{align*}
for $\mathbf{t}=(t_{11},t_{22},t_{32},t_{33},t_{34},t_{41},t_{43},t_{44})$ in sufficiently small ball $B$ centered in $0\in\C^8$. Using the holomorphic coordinates (\ref{hol-coord}), it is a computation to show that $\phi$ satisfies Maurer-Cartan equation. As a side note, thanks to \cite[Theorem 1.1]{CFP06},  we point out $\phi(\mathbf{t})$ parametrizes a locally complete family of complex analytic deformations. We construct the segment $\gamma\colon(-\epsilon,\epsilon)\rightarrow B$, where
\[
t\mapsto\gamma(t)= t (a_{11},a_{22},a_{32},a_{33},a_{34},a_{41},a_{43},a_{44}),
\]
with $(a_{11},a_{22},a_{32},a_{33},a_{34},a_{41},a_{43},a_{44})
\in\C^8$. The corresponding curve of deformations is
\begin{align*}
t\mapsto \phi(t)&=t(a_{11}\overline{\eta}^{1}\otimes Z_{1}+a_{22}\overline{\eta}^{2}\otimes Z_{2}+a_{32}\overline{\eta}^{2}\otimes Z_{3}+a_{33}\overline{\eta}^{3}\otimes Z_{3}\\
&+a_{34}\overline{\eta}^{4}\otimes Z_{3}+a_{41}\overline{\eta}^{1}\otimes Z_{4}+a_{43}\overline{\eta}^{3}\otimes Z_{4}+a_{44}\overline{\eta}^{4}\otimes Z_{4})
\end{align*}
whose derivative in $t=0$ is clearly
\begin{align*}
\phi'(0)&=a_{11}\overline{\eta}^{1}\otimes Z_{1}+a_{22}\overline{\eta}^{2}\otimes Z_{2}+a_{32}\overline{\eta}^{2}\otimes Z_{3}+a_{33}\overline{\eta}^{3}\otimes Z_{3}\\
&+a_{34}\overline{\eta}^{4}\otimes Z_{3}+a_{41}\overline{\eta}^{1}\otimes Z_{4}+a_{43}\overline{\eta}^{3}\otimes Z_{4}+a_{44}\overline{\eta}^{4}\otimes Z_{4}.
\end{align*}
Via structure equations (\ref{eq:RT2_struct_eq}) and the expression of $\phi'(0)$, we obtain that
\begin{align}\label{eq:RT2_lfhside}
2i\mathfrak{Im}&((\del\circ i_{\phi'(0)}\circ\del)(\omega))=\\
&\frac{1}{8}(i \alpha_{3\overline{3}}(a_{34}+\overline{a}_{34})+i\alpha_{4\overline{4}}(a_{43}+\overline{a}_{43})+\alpha_{3\overline{4}}(a_{33}+\c{a}_{44})-\overline{\alpha}_{3\overline{4}}(a_{44}+\overline{a}_{33}))\,\,\eta^{12\overline{12}}.\nonumber
\end{align}
We observe that $\eta^{12\overline{12}}=\frac{1}{4}\del\delbar(\eta^{3\overline{4}})$, therefore the real $(2,2)$-form $\eta^{12\overline{12}}$ represents the vanishing class in $H_{BC}^{2,2}(M)$. Hence, Corollary \ref{cor:main} gives no obstruction.

Nonetheless, if we take any smooth curve of SKT Hermitian invariant metrics $\{\omega_t\}$ along $\phi(t)$ such that $\omega_0=\omega$, written as $\omega_t=e^{i_{\phi(t)}|i_{\overline{\phi(t)}}}(\omega(t))$ with 
\[
\omega(t)=\frac{i}{2}\sum_{j=1}^4\alpha_{j\overline{j}}(t)\,\,\eta^{j\overline{j}}+\frac{1}{2}\sum_{j<k}\left(\alpha_{j\overline{k}}(t)\,\,\eta^{j\overline{k}}-\overline{\alpha}_{j\overline{k}}(t)\,\,\eta^{k\overline{j}}\right),
\]
a straightforward computation yields
\[
\del\delbar \omega'(0)=\frac{1}{8}\mathfrak{Im}({\alpha'_{3\overline{4}}}(0))\eta^{12\overline{12}},
\]
therefore, by imposing equation (\ref{eq:main}) of Theorem \ref{thm:main}, we obtain the following result.
\begin{thm}
Let $(M,J,g,\omega)$ be the nilmanifold obtained as the compact quotient $\Gamma/G$ of the Lie group $G:=\mathbb{H}(3;\R)\times\mathbb{H}(3;\R)\times \R^2$ by a lattice $\Gamma$ of $G$, with complex structure $J$ defined through the invariant coframe of $(1,0)$-complex forms $\{\eta^1,\eta^2,\eta^3,\eta^4\}$ with structure equations
\begin{align*}
\begin{cases}
&d\eta^1=0,\quad d\eta^2=0,\\
&d\eta^3=\frac{1}{2}\eta^{1\overline{1}},\\ 
&d\eta^4=\frac{1}{2}\eta^{2\overline{2}}.
\end{cases}
\end{align*}
Let us consider the curve of deformations
\begin{align*}
t\mapsto \phi(t)&=t (a_{11}\overline{\eta}^{1}\otimes Z_{1}+a_{22}\overline{\eta}^{2}\otimes Z_{2}+a_{32}\overline{\eta}^{2}\otimes Z_{3}+a_{33}\overline{\eta}^{3}\otimes Z_{3}+\\
&+a_{34}\overline{\eta}^{4}\otimes Z_{3}+a_{41}\overline{\eta}^{1}\otimes Z_{4}+a_{43}\overline{\eta}^{3}\otimes Z_{4}+a_{44}\overline{\eta}^{4}\otimes Z_{4}),\quad t\in(-\epsilon,\epsilon)
\end{align*}
and any smooth curve of Hermitian invariant metrics $\{\omega_t\}_{t\in(-\epsilon,\epsilon)}$ along $\phi(t)$ such that $\omega_0=\omega$, with $\omega_t=e^{i_{\phi(t)}|i_{\overline{\phi(t)}}}(\omega(t))$, where
\[
\omega(t)=\frac{i}{2}\sum_{j=1}^4\alpha_{j\overline{j}}(t)\,\,\eta^{j\overline{j}}+\frac{1}{2}\sum_{j<k}\left(\alpha_{j\overline{k}}(t)\,\,\eta^{j\overline{k}}-\overline{\alpha}_{j\overline{k}}(t)\,\,\eta^{k\overline{j}}\right).
\]
Then a necessary condition for $\omega_t$ to be SKT for any $t\in(-\epsilon,\epsilon)$
is that
\[
i \alpha_{3\overline{3}}(a_{34}+\overline{a}_{34})+i\alpha_{4\overline{4}}(a_{43}+\overline{a}_{43})+\alpha_{3\overline{4}}(a_{33}+\c{a}_{44})-\overline{\alpha}_{3\overline{4}}(a_{44}+\overline{a}_{33})=\mathfrak{Im}({\alpha'_{3\overline{4}}}(0)).
\]
\end{thm}

Authors state no conflict of interest.

\end{document}